\documentclass[11pt]{amsart}
\usepackage{xcolor}
\usepackage{graphicx}
\newtheorem{thm}{Theorem}[section]

\newtheorem{lem}[thm]{Lemma}
\newtheorem{prop}[thm]{Proposition}
\theoremstyle{definition}
\newtheorem{defn}[thm]{Definition}
\theoremstyle{remark}
\newtheorem{rem}[thm]{Remark}
\numberwithin{equation}{section}

\usepackage{hyperref}
\hypersetup{
   colorlinks=true,
   linkcolor=blue,  
   urlcolor=cyan,
}

\def\ch{\mathcal{ H}}

\def\bn{{\mathbb N}}

\def\a{\alpha}
\def\b{\beta}
\def\g{\gamma}  
\def\d{\delta}

\def\l{\lambda}

\def\tr{\mathop{\rm Tr}}
\def\id{{\bf 1}\!\!{\rm I}}

\def\<{\langle}
\def\>{\rangle}
\def\a{\alpha}

\begin{document}

\title[$k$-Kadison-Schwarz mappings]{
Constructing $k$-Kadison-Schwarz maps}

\author{Farrukh Mukhamedov}
\address{Farrukh Mukhamedov\\
 Department of Mathematical Sciences\\
College of Science, The United Arab Emirates University\\
P.O. Box, 15551, Al Ain\\
Abu Dhabi, UAE} \email{{\tt far75m@gmail.com} {\tt
farrukh.m@uaeu.ac.ae}}

\author{Dariusz Chru\'sci\'nski}
\address{Dariusz Chru\'sci\'nski\\
Institute of Physics, Faculty of Physics, Astronomy and
Informatics, Nicolaus Copernicus University, Grudzi\c{a}dzka 5/7,
87--100 Toru\'n, Poland} \email{{\tt darch@fizyka.umk.pl}}

\begin{abstract}

We study $k$-Kadison–Schwarz ($k$-KS) mappings on matrix algebras and
derive explicit conditions ensuring the $k$-KS property
for two classes of maps parameterized by a single $k$-positive map $\Phi$. Our construction shows how to {\em upgrade} $k$-positivity to $k$-KS property.  We also introduce KS-decomposability and provide conditions under which positive maps admit a convex decomposition into KS and co-KS components. 
\end{abstract}

\maketitle

\section{Introduction}

Positive linear maps between matrix algebras play a fundamental role in both operator algebras \cite{Paulsen,Stormer,Kadison,Bhatia,bengtsson-zyczkowski} and quantum information theory  \cite{QIT}. In the algebraic formulation of quantum mechanics, states are described by positive normalized functionals on $C^*$--algebras, and physical processes are represented by positive maps between these algebras. In the finite-dimensional setting, where one works with matrix algebras $B(H_d)$ over a $d$-dimensional Hilbert space $H_d$, positivity can be refined in several ways, most notably by the concepts of $k$--positivity and complete positivity. Completely positive (CP) maps provide the proper mathematical description of physically implementable quantum channels, and their structure is elegantly encoded in the Choi--Jamio{l}kowski isomorphism, which identifies CP maps with positive semidefinite operators on $H_d \otimes H_d$. Through this correspondence, positive but not completely positive maps give rise to entanglement witnesses and hence constitute a key tool in the study of quantum correlations.



A unital linear map $\Phi : B(H_d) \to B(H_d)$ is said to satisfy the Kadison--Schwarz inequality if
\begin{equation}\label{eq:KS}
  \Phi(X^* X)\geq\Phi(X)^* \Phi(X),
  \qquad X \in B(H_d).
\end{equation}
Maps fulfilling \eqref{eq:KS} will be called KS operators. Every unital completely positive map defines a KS operator. However, the converse is not true, that is, the KS property is strictly stronger than positivity and strictly weaker than complete positivity \cite{Ev,Tomiyama1985,Choi-1980}. Kadison--Schwarz maps have appeared in various contexts---for example, in the analysis of dynamical maps and divisibility, and in connections to Jordan structures and operator inequalities---their systematic study in higher dimensions is still relatively limited compared to the rich theory of completely positive maps \cite{Rob,majewski-2001,DM2019,Majewski-2019,DM20,Sun1,Alex,Amato,GF,CB24,ZM25,vomEnde}.

For $k\in\{1,\dots,d\}$ we say that a unital map $\Phi$ is a \emph{$k$--KS operator} if the aplification $\Phi_k :=\mathrm{id}_k \otimes \Phi$ satisfies the Kadison--Schwarz inequality on $M_k(B(H_d))$. Now,
any unital $k$-positive map defines $(k-1)$-KS operator \cite{Tomiyama1985,Ev}. In this paper we show how to {\em upgrade} a $k$-positive unital and trace-preserving map $\Phi$ to a $k$-KS operator.  Given a map $\Phi$ we analyze the following classes of maps

\begin{equation}
    \Lambda^-_{a}(X)
  \;=\;
  \frac{1}{d-a}\,\big( \tr(X)\,\id_d - a\Phi(X) \big),
\end{equation}
and

\begin{equation}
    \Lambda^+_{a}(X)
  =  \frac{a}{d}\,\tr(X)\,\id_d + (1-a)\Phi(X),
\end{equation}
and provide conditions upon the parameter  $a \in \mathbb{R}$ which do guarantee that the above map $\Lambda^\pm_a$ are $k$-KS. In the special case when $\Phi$ is an identity or transposition map the necessary and sufficient conditions for $k$-KS property were derived by Tomiyama \cite{Tomiyama1985}. Hence, our result may be considered as a generalization of \cite{Tomiyama1985}. Note, however, that contrary to \cite{Tomiyama1985} we provide only sufficient conditions for $k$-KS.

In analogy to decomposable positive maps we introduce the notion of \emph{KS--decomposability}. A unital positive map $\Phi:B(H_d)\to B(H_d)$ is called KS--decomposable if it can be written as a convex combination
\[
  \Phi = \lambda \Phi_1 + (1-\lambda)\Phi_2,\qquad 0\le \lambda \le 1,
\]
where $\Phi_1$ is KS and $\Phi_2$ is co-KS, meaning $ \Phi(X^* X)\geq\Phi(X) \Phi(X)^*$. We show that this property is strictly stronger than standard decomposability and derive a refined operator inequality for KS--decomposable maps which strengthens a classical result of St{o}rmer for positive maps with $\Phi(\id_d)\le \id_d$. As an application, we prove that if $\Phi$ is an identity map or transposition, then $\Lambda^\pm_a$ are KS--decomposable in suitable parameter regimes.

The paper is organized as follows. In Section~2 we collect notation and basic facts about positive, completely positive and KS maps on $B(H_d)$. Section~3 provides sufficient conditions for $k$-KS within families $\Lambda^\pm_a$.   Section~4 discusses the very notion of KS--decomposability.
Final comments are collected in Section~5.

\section{Preliminaries}

Consider a finite dimensional Hilbert space  $\ch_d$, ${\rm dim}{\ch} = d$, and the set $B(\ch_d)$
    of operators acting on ${\ch}_d$.  For positive semidefinite
    operators $R$ we write $R \geq 0$. Denote a set of
    all positive semidefinite operators by $B(\ch_d)_+$.
    A linear map $\Phi$ is positive if $\Phi(B(\ch_d)_+) \subset B(\ch_d)_+$. By ${\id_d}$ denote the
    identity transformation on $B(\ch_d)$.
     The linear map
    $\Phi$ is called $k$-positive if the map $\Phi_k := {\rm id}_k \otimes \Phi$
    is positive.  A linear
    map $\Phi$ is called completely positive if it is \textit{$k$-positive} for
    all $k \in \bn$.  $\Phi$ is decomposable if
$\Phi = \Phi_1 + \Phi_2$,     where $\Phi_1$ is completely positive, and $\Phi_2$ completely copositive, i.e. completely positive composed with transposition map.

\begin{defn} Let $\Phi:B(\ch_d)\to B(\ch_d)$ be a unital linear mapping. Then it satisfies
\begin{itemize}
\item[1.] \textit{Kadison-Schwarz (KS) condition} if
\begin{equation}\label{KS1}
\Phi(X^*X)\geq \Phi(X)^*\Phi(X), \ \ \forall X\in B(\ch_d).
\end{equation}
\item[2.]  \textit{co-Kadison-Schwarz (co-KS) condition} if
\begin{equation}\label{cKS1}
\Phi(X^*X)\geq \Phi(X)\Phi(X)^*, \ \ \forall X\in B(\ch_d).
\end{equation}
\end{itemize}
\end{defn}

 In what follows, a mapping satisfying KS-condition will be
called KS-operator. Note that every unital 2-positive map is KS-map, and a seminal  result of Kadison states that any positive unital map
satisfies the inequality \eqref{KS1} for all self-adjoint elements
$X$. By $\mathcal{KS}$ we denote the set of all KS-operators mapping
from $B(\ch_d)$ to $B(\ch_d)$.

\begin{rem} Woronowicz shown \cite{woronowicz-1976} that any unital decomposable positive map satisfies so-called strong Kadison inequality, that is, for every $Y $ such that both  $Y\geq X^*X$ and $Y \geq XX^*$ one has

\[ \Phi(Y) \geq \Phi(X)^* \Phi(X) \ , \ \ {\rm and} \ \  \Phi(Y) \geq \Phi(X) \Phi(X)^* . \]
Woronowicz conjectured \cite{woronowicz-1976} that every unital positive map satisfies the strong Kadison inequality. Kirchberg shown that Woronowicz conjecture is not true \cite{Kirchberg}, see also \cite{Tang} for an explicit counterexample.
\end{rem}

\begin{thm}\label{ks-s}\cite{MA2010} The following assertions hold true:
\begin{enumerate}
\item[(i)] the set $\mathcal{KS}$ is convex; \item[(ii)] Let $U,V$
be unitaries in $B(\ch_d)$, then for any $\Phi\in \mathcal{KS}$
the mapping $\Psi_{U,V}(X)=U\Phi(VXV^*)U^*$ belongs to
$\mathcal{KS}$.
\end{enumerate}
\end{thm}

\begin{defn}

A unital positive mapping $\Phi:B(\ch_d)\to B(\ch_d)$ is called
\textit{KS-decomposable} if there exist a unital KS and co-KS
operators $\Phi_1$ and $\Phi_2$, respectively, such that
$$
\Phi=\l\Phi_1+(1-\l)\Phi_2
$$
for some $\l\in[0,1]$.

\end{defn}

\begin{defn}
Let $\Phi:B(\ch_d)\to B(\ch_d)$ be a unital linear map and
let $k \in \{1,\dots,d\}$. We say that $\Phi$ is a \emph{$k$--Kadison--Schwarz
operator} (or simply a \emph{$k$--KS operator}) if the amplified map
\[
  \Phi_k := \operatorname{id}_k \otimes \Phi : M_k(B(\ch_d)) \longrightarrow M_k(B(\ch_d))
\]
satisfies the Kadison--Schwarz inequality, that is,
\[
  \Phi_k(X^{*}X)
  \geq
  \Phi_k(X)^{*}
  \Phi_k(X),
  \qquad \forall\, X \in M_k(B(\ch_d)).
\]
\end{defn}
Equivalently, a unital map $\Phi$ is $k$-KS iff for any set of operators $\{X_1,\ldots,X_k\}\in B(\ch_d)$ one has

\begin{equation}
    \sum_{i,j=1}^k E_{ij} \otimes\Big[\Phi(X_i^* X_j) - \Phi(X_i)^* \Phi(X_j) \Big] \geq 0
\end{equation}
where $E_{ij}$ are matrix units in $M_k$.  In particular a unital map $\Phi$ is 2-KS  and only if
\begin{equation}\label{2KS2}
\left(
      \begin{array}{cc}
        \Phi(A^*A)  & \Phi(A^*B) \\
       \Phi(B^*A) & \Phi(B^*B) \\
         \end{array}
    \right)\geq
    \left(
      \begin{array}{cc}
        \Phi(A)^*\Phi(A)  & \Phi(A)^*\Phi(B) \\
       \Phi(B)^*\Phi(A) & \Phi(B)^*\Phi(B) \\
         \end{array}
    \right)
\end{equation}
for any $A,B \in B(\ch_d)$. Interestingly, one proves

\begin{thm}\cite{Ev}
Let $\Phi: B(\ch_d)\to B(\ch_d)$ be unital
$(k+1)$-positive. Then $\Phi$ is $k$-KS-operator.
\end{thm}

The vector space $B(\ch_d)$  endowed  with the Hilbert-Schmidt inner product
$(X,Y)_{HS}:=\tr(X^*Y)$
becomes a Hilbert space $\mathcal H_{HS}$. The corresponding  Hilbert-Schmidt norm reads $\|X\|_{HS} := \sqrt{\tr\, X^*X}$. For any linear map $\Phi : B(\ch_d) \to B(\ch_d)$ one defines its dual $\Phi^* : B(\ch_d) \to B(\ch_d)$ via

\begin{equation}\label{}
  (\Phi^*(A),B)_{HS} = (A,\Phi(B))_{HS} ,
\end{equation}
for any $A,B \in B(\ch_d)$. Notice that $\Phi$ is trace-preserving if and only if $\Phi^*$ is unital.
A linear map $\Phi : B(\ch_d) \to B(\ch_d)$. A linear mapping $\Phi : B(\ch_d) \to B(\ch_d)$ is  called \textit{Hilbert--Schmidt contraction}, if

\begin{equation}
    \|\Phi(X)\|_{HS}\leq \|X\|_{HS} ,
\end{equation}
for all $X\in B(\ch_d)$. Let us introduce the following completely positive, trace-preserving and unital map $\Delta : B(\ch_d) \to B(\ch_d)$

\[  \Delta(X) = \frac 1d \id_d {\rm Tr}\,X . \]
In quantum information it represents completely depolarizing quantum channel. Note that $\Phi$ is a Hilbert-Schmidt contraction iff

\begin{equation}   \label{Delta-HS}
    \Delta(\Phi(X)^* \Phi(X)) \leq \Delta(X^*X) .
\end{equation}
Consider now the amplification

\[   \Delta_k := {\rm id}_k \otimes \Delta :  M_k(B(\ch_d))\to M_k(B(\ch_d))  , \]
that is,

\[  \Delta_k(X) = \frac 1d {\rm Tr}_2 X \otimes  \id_d \, \]
where the partial trace is defined as follows: for any $X = M \otimes Y \in M_k(B(\ch_d)) $

\[   {\rm Tr}_2 X = M\,  {\rm Tr}\, Y \in M_k . \]

\begin{lem}\label{3KS-T1}
For any $k=1,2,\ldots$ the mapping $\Delta_k$ has the following properties:\\
\begin{itemize}
  \item[(i)] $\Delta_k \, \Delta_k =  \Delta_k $, that is $\Delta_k$ is a projector,  \\
  \item[(ii)] $\Delta_k(\Delta_k(X)^*\Delta_k(X))=\Delta_k(X)^*\Delta_k(X)$ for all $X\in M_k(B(\ch_d))$; \\
  \item[(iii)] if $\Delta_k(Y)=0$, then $\Delta_k(Y\Delta_k(X))=0$  for all $X\in M_k(B(\ch_d))$.
\end{itemize}
\end{lem}

In what follows we generalize (\ref{Delta-HS}) and consider linear maps $\Phi : B(\ch_d) \to B(\ch_d) $ which satisfy the following property

\begin{equation}  \label{Phi-k}
     \Delta_k(\Phi_{k}(X)^*\Phi_{k}(X)) \leq \Delta_k(X^*X) ,
\end{equation}
where $\Phi_{k} :=\mathrm{id}_k\otimes \Phi$.

\begin{prop}
Let $\Phi : B(\ch_d) \to B(\ch_d)$ be a unital and trace-preserving linear map.
If $\Phi$ is a $k$-KS operator, then $\Phi$ satisfies (\ref{Phi-k}).
\end{prop}

\begin{proof} since $\Phi$ is $k$-KS one has

\[
\Phi_{k}(X^*X) \geq \Phi_{k}(X)^* \Phi_{k}(X)
 ,
\]
{for all } $X \in M_k(B(\ch_d))$,  and hence

\[
\Delta_k(\Phi_{k}(X^*X)) \geq \Delta_k( \Phi_{k}(X)^* \Phi_{k}(X) ) .
\]
Observe that since $\Phi$ is trace-preserving

\[ \Delta_k(\Phi_{k}(X^*X))  = ({\rm id}_k \otimes {\rm Tr})({\rm id}_k \otimes \Phi)(X^*X) \otimes \id_d  =
({\rm id}_k \otimes {\rm Tr})(X^*X) \otimes \id_d = \Delta_k(X^*X) ,
\]
which completes the proof.
\end{proof}

\begin{rem}
    If  $\Phi={\rm id}$, then it is evident, that $\Phi$ is $k$-KS operator and hence satisfies (\ref{Phi-k}).
Interestingly, transposition map $\Phi = T$ is not even 2-positive but it satisfies (\ref{Phi-k}) as well.
\end{rem}

\begin{thm}\label{HS}
Let $\Phi$ be a Hilbert--Schmidt contraction. Then for every $k \in \mathbb N$, the map $\Phi_k$ satisfies (\ref{Phi-k}).
\end{thm}
For the proof cf. Appendix. The above Theorem shows that it is not difficult  to construct maps which satisfy (\ref{Phi-k}).

\section{Class of $k$-Kadison-Schwarz operators}

In this section, we analyze two classes of $k$-Kadison-Schwarz operators.

\subsection{Generalized Reduction Maps}

Let us consider the following so-called reduction mapping
\begin{equation}\label{F0}
R_a(X)=\frac{1}{d-a}(\id_d {\rm Tr} X  - a \, X ), \quad X\in B(\ch_d).
\end{equation}
In what follows, we frequently use the following fact.

\begin{prop}
    \label{L1}\cite{Tomiyama1985} The mapping $R_a$ is $k$-positive iff $a\leq 1/k$ ($k\leq d$).
\end{prop}

%


Now, let $\Phi:B(\ch_d)\to B(\ch_d)$ be a positive unital, trace-preserving map satisfying (\ref{Phi-k}). We define the following mapping
\begin{equation}\label{FS1}
\Lambda^-_{a}=\frac{1}{d-a}(d \Delta -a \Phi)
\end{equation}
It is clear that $\Lambda^-_{a}$ is a unital trace-preserving mapping.
Consider now the corresponding amplification

\[   \Lambda^-_{a;k} := {\rm id}_k \otimes \Lambda^-_a = \frac{1}{d-a}\big(d\,\Delta_k -a\, \Phi_{k}\big) .\]


\begin{thm}\label{T1} Assume $\Phi$ is $k$-positive, unital and satisfies (\ref{Phi-k}). Then the mapping $\Lambda^-_{a}$ defines a $k$-KS operator if
\begin{equation}\label{4KS2}
0\leq a\leq\frac{d}{kd +1}.
\end{equation}
\end{thm}

\begin{proof}

One has
\begin{eqnarray*}
&&\Lambda^-_{a;k}(X)^* \Lambda^-_{a;k}(X) = \frac{1}{(d-a)^2}\bigg(d^2 \Delta_k(X)^*\Delta_k(X)-a d \big(\Delta_k(X^*)\Phi_{k}(X)+\Phi_{k}(X)^*\Delta_k(X)\big)+a^2 \Phi_{k}(X)^*\Phi_{k}(X)\bigg)\\[2mm]
&&\Lambda^-_{a;k}(X)(X^*X)=\frac{1}{d-a}(d \Delta_k(X^*X)-a \Phi_{k}(X^*X))
\end{eqnarray*}
Therefore, the KS-condition for $\Lambda^-_{a;k}$ reads as follows
\begin{eqnarray}\label{4KS1}
(d-a)(d\Delta_k(X^*X)-a \Phi_{k}(X^*X))&\geq & d^2 \Delta_k(X)^*\Delta_k(X)- ad\big(\Delta_k(X^*)\Phi_{k}(X)+\Phi_{k}(X)^*\Delta_k(X)\big)\nonumber \\
&+& a^2 \Phi_{k}(X)^*\Phi_{k}(X)
\end{eqnarray}
Assume that \eqref{4KS2} holds. To prove that $\Lambda^-_{a}$ is a $k$-KS operator, it is enough to establish
\eqref{4KS1}.

Now consider two separate cases.


{\tt Case 1.} Suppose that $\Delta_k(X)=0$. Then from \eqref{4KS1} it follows that
\begin{eqnarray}\label{4KS3}
(d-a)d\Delta_k(X^*X)-a(d-a)\Phi_{k}(X^*X)-a^2 \Phi_{k}(X)^*\Phi_{k}(X)\geq 0.
\end{eqnarray}
We rewrite the last one as follows
\begin{eqnarray*}
\frac{(d-a)}{da}\Delta_k(a(d-a)X^*X+a^2X^*X)-\big(a(d-a)\Phi_{k}(X^*X)+a^2\Phi_{k}(X^*)\Phi_{k}(X)\big)\geq 0.
\end{eqnarray*}
The trace preserving condition of $\Phi$ yields  $\Delta_k(\Phi_{k}(X))=\Delta_k(X)$, therefore,
\begin{eqnarray*}
\frac{(d-a)}{da}\bigg(\Delta_k(a(d-a)\Phi_{k}(X^*X))+a^2\Delta_k(X^*X)\bigg)-\big(a(d-a)\Phi_{k}(X^*X)+a^2\Phi_{k}(X^*)\Phi_{k}(X)\big)\geq 0.
\end{eqnarray*}
Now, by means of $\Delta_k$-KS-condition , one gets
$$
\Delta_k(a(d-a)\Phi_{k}(X^*X))+a^2\Delta_k(X^*X)\geq \Delta_k(a(d-a)\Phi_{k}(X^*X))+a^2\Delta_k(\Phi_{k}(X^*)\Phi_{k}(X))
$$
Therefore, it is enough to
show
\begin{eqnarray}\label{KSw}
&&\frac{(d-a)}{da}\bigg(\Delta_k(a(d-a)\Phi_{k}(X^*X))+a^2\Delta_k(\Phi_{k}(X^*)\Phi_{k}(X))\bigg)\nonumber\\
&&-\big(a(d-a)\Phi_{k}(X^*X)+a^2\Phi_{k}(X^*)\Phi_{k}(X)\big)\geq 0.
\end{eqnarray}
Due to $k$-positivity of $\Phi$, it is clear that
$$
Z=a(d-a)\Phi_{k}(X^*X))+a^2\Phi_{k}(X^*)\Phi_{k}(X)\geq 0
$$
Hence, \eqref{KSw} means
$
\frac{(d-a)}{da}\Delta_k(Z)-Z\geq 0
$
which due to Lemma \ref{L1}, is satisfied iff
$\frac{da}{d-a} \leq \frac{1}{k}$ and hence
$$
a\leq \frac{d}{kd+1}.
$$

{\tt Case 2.}  Now assume that $\Delta_k(X)\neq 0$, and denote
$$
Y=\Delta_k(X)-X.
$$
It is clear that $\Delta_k(Y)=0$ and $X=\Delta_k(X)-Y$.
Then
\begin{eqnarray}\label{4KS5}
X^*X=\Delta_k(X)^*\Delta_k(X)-(\Delta_k(X^*)Y+\Delta_k(X)Y^*)+Y^*Y
\end{eqnarray}
Hence, by Lemma \ref{3KS-T1}, one gets
\begin{eqnarray}\label{X1}
\Delta_k(X^*X)&=&\Delta_k(X)^*\Delta_k(X)+\Delta_k(Y^*Y).
\end{eqnarray}
Now, unitality of $\Phi$ implies $\Phi_{k}(\Delta_k(X))=\Delta_k(X)$. Therefore,
\begin{eqnarray*}\label{X2}
\Delta_k(X^*)\Phi_{k}(X)+\Phi_{k}(X^*)\Delta_k(X)&=&{2}\Delta_k(X)^*\Delta_k(X)-(\Delta_k(X^*)\Phi_{k}(Y)+\Phi_{k}(Y)^*\Delta_k(X))\\[2mm]
\Phi_{k}(X^*)\Phi_{k}(X)&=&\Delta_k(X^*)\Delta_k(X)-(\Delta_k(X^*)\Phi_{k}(Y)+\Delta_k(X)\Phi_{k}(Y^*))+\Phi_{k}(Y^*)\Phi_{k}(Y)
\end{eqnarray*}
Hence, substituting all these to \eqref{4KS1} and after simple calculations, we arrive at
%
\begin{eqnarray*}
(d-a)d\Delta_k(Y^*Y)-a(d-a)\Phi_{k}(Y^*Y)-a^2\Phi_{k}(Y)^* \Phi_{k}(Y)\geq 0.
\end{eqnarray*}
which, due to the case 1, is satisfied, if \eqref{4KS2} holds.
This completes the proof.
\end{proof}

\begin{rem}\label{C1} If $\Phi$ is an identity map it was proved by Tomiyama \cite{Tomiyama1985} that $\Lambda^-_a$ is $k$-KS iff

$$
0\leq a\leq\frac{d}{kd+1} ,
$$
that is, the above condition is both necessary and sufficient in this case.
\end{rem}

\bigskip

\subsection{Generalized Depolarized Maps}

Let $\Phi:B(\ch_d)\to B(\ch_d)$ be a positive UTP $\Delta_k$-KS-operator.  We define a new mapping
 $\Lambda^+_{a}: B(\ch_d)\to B(\ch_d)$  by
\begin{equation}\label{PS1}
\Lambda^+_{a} = a \Delta +(1-a) \Phi .
\end{equation}
It is clear that $\Lambda^+_a$ is a unital trace-preserving mapping.
The corresponding amplification reads

\[  \Lambda^+_{a;k} = a \Delta_k +(1-a) \Phi_k . \]

\begin{thm}\label{KK}  Assume $\Phi$ is $k$-positive, unital and satisfies (\ref{Phi-k}).  Then the mapping $\Lambda^+_a$ is a $k$-KS operator if
\begin{equation}\label{KK1}
1 + \frac{\sqrt{4kd+1}+1}{2kd} \geq a\geq 1 - \frac{\sqrt{4kd+1}-1}{2kd}.
\end{equation}
\end{thm}

\begin{proof} Assume that \eqref{KK1} holds. To prove that $\Lambda^+_{a}$ is a $k$-KS operator, it is enough to establish that $\Lambda^+_{a;k}$ is KS-operator. It is enough to consider the case $\Delta_k(X)=0$. One finds
\small
\begin{eqnarray*}
&& \Lambda^+_{a;k}(X^*X) - \Lambda^+_{a;k}(X)^* \Lambda^+_{a;k}(X)=\frac{a}{d}\Delta_k(X^*X)+(1-a)\Phi_{k}(X^*X)-(1-a)^2\Phi_{k}(X)^*\Phi_{k}(X)\\[2mm]
&&={a}\bigg(\Delta_k(X^*X)-\frac{1}{dk}\Phi_{k}(X)^*\Phi_{k}(X)\bigg)+(1-a)\Phi_{k}(X^*X)+\bigg(\frac{a}{kd}-(1-a)^2\bigg)\Phi_{k}(X)^*\Phi_{k}(X)\\[2mm]
&&\geq {a}\bigg(\Delta_k(\Phi_{k}(X)^*\Phi_{k}(X))-\frac{1}{dk}\Phi_{k}(X)^*\Phi_{k}(X)\bigg)+(1-a)\Phi_{k}(X^*X)+\bigg(\frac{a}{kd}-(1-a)^2\bigg)\Phi_{k}(X)^*\Phi_{k}(X)\\[2mm]
\end{eqnarray*}
\normalsize
By Proposition \ref{L1} we have
$$
\Delta_k(\Phi_{k}(X)^*\Phi_{k}(X))-\frac{1}{dk}\Phi_{k}(X)^*\Phi_{k}(X)\geq 0.
$$
Therefore,
$$
\Lambda^+_{a;k}(X^*X) - \Lambda^+_{a;k}(X)^* \Lambda^+_{a;k}(X)\geq 0
$$
if
$$
\frac{a}{kd}-(1-a)^2\geq 0.
$$
Solving the last one, we arrive at \eqref{KK1}.
This completes the proof.
\end{proof}


%
%
%

\begin{rem}\label{C2} If $\Phi$ is a transposition map it was proved by Tomiyama \cite{Tomiyama1985} that $\Lambda^+_a$ is KS iff

$$
1 + \frac{1}{d-1} \geq a\geq 1 - \frac{\sqrt{4d+1}-1}{2d} ,
$$
and it is 2-KS iff it is 2-positive.
\end{rem}

\section{KS-decomposability}

In this section we study KS-decomposability of the
reduction and depolarizing mappings. Recall that a unital positive mapping $\Phi:B(\ch_d)\to
B(\ch_d)$ is called \textit{KS-decomposable} if there exist a
unital KS and co-KS operators $\Phi_1$ and $\Phi_2$, respectively,
such that
$$
\Phi=\lambda\Phi_1+(1-\lambda)\Phi_2
$$
for some $\lambda\in[0,1]$.

\begin{thm}\label{KS-d} Let $\Phi:B(\ch_d)\to B(\ch_d)$ be a KS-decomposable mapping,
i.e. there exist a unital KS and co-KS operators $\Phi_1$ and
$\Phi_2$, respectively, such that
\begin{equation}\label{KS-d0}
 \Phi=\lambda\Phi_1+(1-\lambda)\Phi_2
\end{equation}
for some $\lambda\in[0,1]$. Then

\begin{equation}  \label{KS-d1}
    \Phi(x^* \circ x) - \Phi(x)^*\circ \Phi(x)  \geq \lambda(1-\lambda) (\Phi_1(x)-\Phi_2(x))^* \circ (\Phi_1(x)-\Phi_2(x)) ,
\end{equation}
where we introduced the Jordan product $A \circ B := AB + BA$.

\end{thm}

\begin{proof} Since $\Phi_1$ is KS-operator, and $\Phi_2$ is co-KS, from \eqref{KS-d0} we obtain
\begin{eqnarray}\label{KS-d2}
\Phi(x^*x)=\lambda\Phi_1(x^*x)+(1-\lambda)\Phi_2(x^*x)\geq
\lambda\Phi_1(x)^*\Phi_1(x)+(1-\lambda)\Phi_2(x)\Phi_2(x)^*.
\end{eqnarray}
Similarly,
\begin{eqnarray}\label{KS-d3}
\Phi(xx^*)\geq\lambda \Phi_1(x)\Phi_1(x)^*+(1-\lambda)\Phi_2(x)^*\Phi_2(x).
\end{eqnarray}
One can see that
\begin{eqnarray}\label{KS-d4}
\Phi(x)^*\Phi(x)&=&\lambda^2\Phi_1(x)^*\Phi_1(x)+\lambda(1-\lambda)\Phi_1(x)^*\Phi_2(x)\nonumber\\
&&+\lambda(1-\lambda)\Phi_2(x)^*\Phi_1(x)+(1-\lambda)^2\Phi_2(x)^*\Phi_2(x),\\[2mm]\label{KS-d5}
\Phi(x)\Phi(x)^*&=&\lambda^2\Phi_1(x)\Phi_1(x)^*+\lambda(1-\lambda)\Phi_1(x)\Phi_2(x)^*\nonumber\\
&&+\lambda(1-\lambda)\Phi_2(x)\Phi_1(x)^*+(1-\lambda)^2\Phi_2(x)\Phi_2(x)^*.
\end{eqnarray}
Hence, combining \eqref{KS-d2}-\eqref{KS-d5}, we obtain
\begin{eqnarray*}
&&\Phi(x^*x)+\Phi(xx^*)-\Phi(x)^*\Phi(x)-\Phi(x)\Phi(x)^*\nonumber\\[2mm]
&\geq&\lambda(1-\lambda)\bigg(\Phi_1(x)^*\Phi_1(x)+\Phi_2(x)^*\Phi_1(x)-\Phi_2(x)^*\Phi_1(x)-\Phi_1(x)^*\Phi_2(x)\bigg)\\[2mm]
&&+\lambda(1-\lambda)\bigg(\Phi_1(x)^*\Phi_1(x)+\Phi_2(x)^*\Phi_1(x)-\Phi_2(x)^*\Phi_1(x)-\Phi_1(x)^*\Phi_2(x)\bigg)\\[2mm]
&=&\lambda(1-\lambda)\bigg((\Phi_1(x)-\Phi_2(x))^*(\Phi_1(x)-\Phi_2(x))+(\Phi_1(x)-\Phi_2(x))(\Phi_1(x)-\Phi_2(x))^*\bigg),
\end{eqnarray*}
which completes the proof.
\end{proof}

\begin{rem} We point out that in \cite{St}, for any positive mapping $\Phi:B(\ch_d)\to B(\ch_d)$ with $\Phi(\id)\leq\id$, it was proved a similar
kind of inequality, i.e.

\[  \Phi(X^* \circ X) \geq \Phi(X)^* \circ \Phi(X) .\]
%
However, if one requires the KS-decomposability of $\Phi$, then we
get a more stronger inequality, i.e. we have a more accurate lower
bound for the above given expression.

\end{rem}

\begin{thm}\label{dKS1} If $\Phi=T$ the mapping $\Psi_a:=\Lambda^+_a$  is KS-decomposable for $a\in[0,1]$.
\end{thm}

\begin{proof} If $a\geq\frac{2d+1-\sqrt{4d+1}}{2d}$, then due to
Theorem \ref{KK}, the mapping $\Psi_a$ is a KS-operator, hence it
is KS-decomposable (since $\lambda=1$ and $\Phi_1=\Psi_a$). Therefore,
we assume that
\begin{equation}\label{dKS2}
a<\frac{2d+1-\sqrt{4d+1}}{2d}.
\end{equation}
Now, we suppose that $\Psi_a$ has the following decomposition:

\begin{equation}\label{dKS3}
\Psi_a=\lambda \Psi_\beta +(1-\lambda)T_d ,
\end{equation}
where $T_d$ stands for transposition in $M_d$, and
$$
\Psi_\beta(X)=\frac{\beta}{d}\tr(X)\id_d +(1-\b)X^T.
$$
If we able to find $\lambda$ and $\beta$ so that $\Psi_\beta$ is a
KS-operator, then the required assertion would be proved, since
$T_d$ is co-completely positive.

From \eqref{dKS3}, we find $\lambda\beta=a$ which yields $\beta=\frac{a}{\lambda}$.
Then, we choose $\lambda\in[0,1]$ so that
$$
a\geq \lambda \, \frac{2d+1-\sqrt{4d+1}}{2d}\ .
$$
This implies that
$$
\beta \geq\frac{2d+1-\sqrt{4d+1}}{2d},
$$
which again according to Theorem \ref{KK} yields that $\Psi_\beta$ is
a KS-operator. This completes the proof.
\end{proof}

Now, we turn our attention to the mapping $R_a$ defined by
\eqref{F0}. Namely,
\begin{equation}\label{df1}
R_a(X)=\frac{1}{d-a}(\tr(X)\id_d -aX), \ \ a\in[0,1]
\end{equation}
which is a KS-operator (see Corollary \ref{C1}) iff $a\leq
\frac{d}{d+1}$.

\begin{thm}\label{1dKS}  The mapping $R_a$  is KS-decomposable.
\end{thm}
For the proof cf. Appendix.

%

\section*{Conclusions}

We analyzed $k$-Kadison–Schwarz maps and KS-decomposable positive maps on matrix algebras. Using the construction of the projection $\Delta_k$ we derived explicit criteria guaranteeing the $k$-KS property and clarified the relation between $k$-positivity and $k$-KS condition.
Within this framework we studied two classes of maps $\Lambda^\pm_a$ depending on a particular $k$-positive map $\Phi$. We derive conditions which guarantee that $\Lambda^\pm_a$ are $k$-KS for arbitrary $\Phi$. In partcular cases when $\Phi$ is an identity map or a transposition map our results perfectly agree with those of Tomiyama \cite{Tomiyama1985}.

We also introduced KS-decomposability, showing that certain maps admit a convex decomposition into KS and co-KS components. The resulting inequalities strengthen classical results such as Stormer’s inequality and provide new insight into the structure of positive maps relevant for entanglement detection and quantum channels.

Several directions for further research are suggested by our results. One natural problem is to characterize KS--decomposability beyond the specific families studied here, for instance for extremal positive maps or maps arising from more general quantum channels. Another interesting avenue is to explore systematically the entanglement-detecting power of $k$--KS and KS--decomposable maps, in particular the strength of the associated families of entanglement witnesses compared with those coming from arbitrary positive maps. More generally, the interplay between $k$-positivity, $k$--KS properties and KS--decomposability may provide additional insight into the fine structure of positive cones in matrix algebras and their applications in quantum information theory \cite{Alex}.

Another interesting problem is to analyze our construction where the map $\Phi$ enjoys additional symmetry. A map $\Phi$ is covariant w.r.t. the unitray representation of the group $G$ on the Hilbert space $\ch_d$
if $U_g \Phi(X) U_g^* = \Phi(U_g X U_g^*)$ for all elements $X$. Similarly, $\Phi$ is conjugate covariant if
$U_g \Phi(X) U_g^* = \Phi(\overline{U}_g X U_g^T)$. Note, that maps considered by Tomiyama \cite{Tomiyama1985} enjoy covariance properties: identity map is covariant w.r.t. the full unitary group $U(d)$, and similarly transposition map is conjugate covariant. Studies of covariant Kadison-Schwarz maps were provided recently in \cite{CB24,Ibort}. It would be interesting to provide details analysis of the construction proposed in this paper assuming that the map $\Phi$ is covariant w.r.t. subgroup of diagonal unitary matrices. Such analysis for completely positive maps was recently provided in \cite{Nechita-1,Nechita-2}.

\appendix

\section{Proof of Theorem \ref{HS}}

Let $\ell^2_k$ be the $k$–dimensional Hilbert space with
orthonormal basis $\{e_1,\dots,e_k\}$ and set
\[
\mathcal L:= \ell^2_k \otimes \mathcal H_{HS}.
\]
Fix $X=(x_{ij})_{i,j=1}^k \in M_k(B(\mathcal H_d))$.
For each $j=1,\dots,k$, define a vector $\xi_j \in \mathcal L$ by
\[
\xi_j := \sum_{m=1}^k e_m \otimes x_{mj}.
\]
A direct computation shows
\[
\langle \xi_i,\xi_j\rangle_{\mathcal L}
=
\sum_{m=1}^k \operatorname{Tr}(x_{mi}^* x_{mj})
=
\big(\operatorname{Tr}_2 X^*X\big)_{ij}.
\]
Thus the matrix $\operatorname{Tr}_2 X^*X \in M_k$ is the Gram matrix of the family $\{\xi_j\}_{j=1}^k$.

By assumption, $\Phi$ is a HS-contraction:
\[
\|\Phi(Z)\|_{HS} \le \|Z\|_{HS}
\quad \forall Z \in \mathcal H_{HS}.
\]
Define
\[
S := I_{\ell^2_k} \otimes \Phi
\]
on $\mathcal L$.
Then $S$ is also a contraction:
\[
\|S\eta\|_{\mathcal L} \leq \|\eta\|_{\mathcal L}
\quad \forall \eta \in \mathcal L.
\]
Now
\[
S\xi_j
=
\sum_{m=1}^k e_m \otimes \Phi(x_{mj}),
\]
and therefore
\[
\langle S\xi_i,S\xi_j\rangle_{\mathcal L}
=
\sum_{m=1}^k
\operatorname{Tr}\big(\Phi(x_{mi})^*\Phi(x_{mj})\big)
=
\operatorname{Tr}\big((\Phi_{k}(X)^*\Phi_{k}(X))_{ij}\big).
\]
Hence the scalar matrix
\[
\Big(
\operatorname{Tr}((X^*X)_{ij})
-
\operatorname{Tr}((\Phi_{k}(X)^*\Phi_{k}(X))_{ij})
\Big)_{i,j}
\]
is the difference of the Gram matrices of
$\{\xi_j\}$ and $\{S\xi_j\}$.
Let $c=(c_1,\dots,c_k)\in\mathbb C^k$ and set
\[
\eta := \sum_{j=1}^k c_j \xi_j.
\]
Then
\begin{equation}\label{SF}
\sum_{i,j=1}^k
\overline{c_i} c_j
\Big(
\operatorname{Tr}((X^*X)_{ij})
-
\operatorname{Tr}((\Phi_{k}(X)^*\Phi_{k}(X))_{ij})
\Big)
=
\|\eta\|_{\mathcal L}^2
-
\|S\eta\|_{\mathcal L}^2.
\end{equation}
Since $S$ is a contraction,
\[
\|\eta\|_{\mathcal L}^2
-
\|S\eta\|_{\mathcal L}^2
\geq 0.
\]
Therefore the above scalar matrix is positive semidefinite.
By \eqref{SF} and the definition of $\Delta_k$, the last one is equivalent to
\[
\Delta_k(X^*X)
-
\Delta_k(\Phi_{k}(X)^*\Phi_{k}(X))
\geq 0
\]
in $M_k(B(\mathcal H_d))$.
This completes the proof.

\section{Proof of Theorem \ref{dKS1}}

If $a\leq\frac{d}{d+1}$, then $R_a$ is KS-decomposable. Therefore,
we assume that
\begin{equation}\label{2dKS}
a>\frac{d}{d+1}.
\end{equation}
Suppose that
\begin{equation}\label{21dKS}
R_a(X)=\a\tr(X)\id_d -\b X+\g\tr(X)\id_d +\d X
\end{equation}
This due to \eqref{df1} implies
\begin{equation}\label{4dKS}
\a+\g=\frac{1}{d-a}, \ \ \b-\d=\frac{a}{d-a}.
\end{equation}
Due to the unitality of $R_a$, we need to impose that
\begin{equation}\label{5dKS}
\a d -\b=\lambda, \ \ \g d+\d=1-\lambda,
\end{equation}
for some $\lambda\in[0,1]$.

From \eqref{21dKS}, we obtain the decomposition
$$
R_a=\lambda\Phi_1+(1-\lambda)\Phi_2,
$$
where
\begin{equation}\label{1dFF1}
\Phi_1(X)=\frac{1}{\lambda}(\a\tr(X) \id_d -\b X), \ \
\Phi_2(X)=\frac{1}{1-\lambda}(\g\tr(X) \id_d +\d X).
\end{equation}
According to \eqref{5dKS} the mappings $\Phi_1$,$\Phi_2$ are
unital. Now our aim to choose the parameters so that $\Phi_1$ is a
KS-operator and $\Phi_2$ is a co-KS.

We first notice that only three equations in
\eqref{4dKS},\eqref{5dKS} are linearly  independent. Therefore, we
assume that the parameter $A$ is free variable. Then, one gets
\begin{eqnarray}\label{7dKS}
\b=\a d-\lambda, \ \ \ \g=\frac{1}{d-a}-\a, \ \  \ \d=\a d-\lambda-\frac{a}{d-a}.
\end{eqnarray}

From \eqref{1dFF1}, one finds
\begin{equation}\label{6dKS}
\Phi_1(X)=\frac{\a}{\lambda}\bigg(\tr(X)\id_d -\frac{\b}{\a}X\bigg), \ \
\Phi_2(X)=\frac{\g}{1-\lambda}\bigg(\tr(X)\id_d +\frac{\d}{\g}X\bigg).
\end{equation}

Assume that $\a>0$, $\b>0$ and
$$
\frac{\b}{\a}\leq \frac{d}{d+1}
$$
which due to Corollary \ref{C1} implies $\Phi_1$ is KS. From
\eqref{7dKS} the last condition is equivalent to
\begin{equation*}
\frac{\a d^2}{d+1}\leq \lambda.
\end{equation*}
Due to $\b>0$ we have $\lambda\leq \a d$. Hence,
\begin{equation}\label{8dKS}
\frac{\a d^2}{d+1}\leq \lambda\leq \a d.
\end{equation}

Now we suppose that $\g>0$ which is equivalent to
\begin{equation}\label{0AC}
\a<\frac{1}{d-a}
\end{equation}
The condition \eqref{2dKS} implies that
$$
\frac{a(d+1)}{d(d-a)}\geq \frac{1}{d-a}>\a
$$
it yields
$$
\a d-\frac{a(d+1)}{d-a}<0.
$$
This implies
$$
\a d^2+\a d-\frac{a(d+1)}{d-a}<\a d^2 \ \ \Leftrightarrow
\a d-\frac{a}{d-a}<\frac{\a d^2}{d+1}
$$
From the last inequality due to \eqref{8dKS} we find
$$
\lambda>\a d-\frac{a}{d-a}
$$
which is equivalent to $\d<0$.
Hence,
$$
\Phi_2(X)=\frac{\g}{1-\lambda}\bigg(\tr(X)\id_d -\frac{\tilde \d}{\g}X\bigg).
$$
where $\tilde \d=-\d$.
We know that the mapping $\Phi_2$ is co-CP iff
$\frac{\tilde \d}{\g}\leq 1$. This condition is equivalent to
$$
\lambda\leq (d-1)\a+\frac{1-a}{d-a}.
$$
It is easy to verify that \eqref{0AC} implies
$$
(d-1)\a+\frac{1-a}{d-a}\leq 1.
$$
Thus,
\begin{equation}\label{1AC}
\frac{\a d^2}{d+1}\leq \lambda\leq \min\bigg\{\a d,
(d-1)\a+\frac{1-a}{d-a}\bigg\}.
\end{equation}

If $a<1$ and $\a\leq \frac{1-a}{d-a}$, then
\begin{equation*}\label{1AC}
\a d\leq (d-1)\a+\frac{1-a}{d-a} \end{equation*} hence, one finds
\begin{equation}\label{2AC}
\frac{\a d^2}{d+1}\leq \l\leq \a d.
\end{equation}
Since $\a d\neq \frac{\a d^2}{d+1}$, we may always find the needed
$\l$ belonging to $(0,1]$. Therefore, one can find (using
\eqref{7dKS}), the parameters for which the mapping $\Phi_1$ is KS
and $\Phi_2$ is co-KS.

Now, if $a=1$, then from \eqref{1AC} we again find \eqref{2AC}
which yields the assertion. This completes the proof.

\section*{Acknowledgments}
	F.M.  thanks to the UAEU UPAR Grant No. G00004962 for support. D.C. was supported by the Polish National Science
Centre project No. 2024/55/B/ST2/01781.


\section*{Conflicts of Interest} The authors declare that they have no conflicts of interest.

\section*{Author's contribution} All authors have equally contributed to the paper.

\end{document}